\documentclass[preprint,12pt]{elsarticle}

\usepackage{amssymb}
\usepackage{amsthm}
\usepackage{amsmath}
\usepackage[utf8]{inputenc}
\usepackage[T1]{fontenc}
\usepackage{lmodern}
\usepackage{dsfont}

\usepackage{lineno}

\newtheorem{thm}{Theorem}[section]
\newtheorem{cor}[thm]{Corollary}
\newtheorem{lem}[thm]{Lemma}
\newtheorem{prop}[thm]{Proposition}
\newtheorem{defi}[thm]{Definition}
\newtheorem{ejm}[thm]{Example}
\newtheorem{obs}[thm]{Observation}
\newtheorem{rem}[thm]{Remark}

\newcommand{\set}[1]{\left\{#1\right\}} 

\newcommand{\N}{\mathbb N}

\newcommand{\R}{\mathbb R}
\newcommand{\X}{\mathbb X}
\newcommand{\Y}{\mathbb Y}
\newcommand{\Z}{\mathbb Z}
\newcommand{\ds}{\displaystyle}

\journal{Journal of Mathematical Analysis and Applications}

\begin{document}

\begin{frontmatter}


 \title{Vector valued piecewise continuous almost automorphic functions and some consequences}


\author[1]{Alan Ch\'avez \corref{cor1}%
}
\ead{ajchavez@unitru.edu.pe}
\author[2]{Lenin Qui\~nones Huatangari}
\ead{lenin.quinones@unj.edu.pe}
\cortext[cor1]{Corresponding author}

\affiliation[1]{organization={OASIS \& GRACOCC research groups, Instituto de Investigaci\'on en Matem\'aticas \& Departamento de Matem\'aticas-FCFYM-  Universidad Nacional de Trujillo},
addressline={Av. Juan Pablo II S/N - Urbanizaci\'on San Andr\'es},
postcode={13011},
city={Trujillo},
country={Per\'u}}

\affiliation[2]{organization={Instituto de Ciencia de Datos, Universidad Nacional de Jaen},
city={Cajamarca},
country={Per\'u}}

%


%





\begin{abstract}
In the present work, for $\mathbb{X}$ a Banach space, the notion of piecewise continuous $\mathbb{Z}$-almost automorphic functions with values in finite dimensional spaces is extended to piecewise continuous $\mathbb{Z}$-almost automorphic functions with values in $\mathbb{X}$. Several properties of this class of functions are provided, in particular it is shown that if $\mathbb{X}$ is a Banach algebra, then this class of functions constitute also a Banach algebra; furthermore, using the theory of $\mathbb{Z}$-almost automorphic functions, a new characterization of compact almost automorphic functions is given. As consequences, with the help of $\mathbb{Z}$-almost automorphic functions, it is presented a simple proof of the characterization of almost automorphic sequences by compact almost automorphic functions; the method permits us to give explicit examples of compact almost automorphic functions which are not almost periodic. Also, using the theory developed here, it is shown that  almost automorphic solutions of differential equations with piecewise constant argument are in fact compact almost automorphic. Finally, it is proved that the classical solution of the $1D$ heat equation with continuous $\mathbb{Z}$-almost automorphic source is also continuous $\mathbb{Z}$-almost automorphic; furthermore, we comment applications to the existence and uniqueness of the asymptotically continuous $\mathbb{Z}$-almost automorphic mild solution to abstract integro-differential equations with nonlocal initial conditions.
\end{abstract}




\end{frontmatter}

\section{Introduction}\label{sec:Introduction}

Continuous almost automorphic functions were introduced by the Mathematician S. Bochner in his work \cite{03}, subsequently the mathematical community were interested in the structural properties of this function space. For some further developments, see for example the works \cite{04,05,18,veech1965almost,milnes1977almost,terras1970almost,
berger2004almost,dai2020almost,hric2014construction} and references cited therein.

Compact almost automorphic functions were studied by the first time by the mathematician Phillip R. Bender in his Ph.D. Thesis \cite{bender1966some}, after that A. M.  Fink in his works \cite{fink1968almost,fink1969extensions} studied existence of compact almost automorphic solutions to differential equations and also studied the extension problem of discrete almost automorphic functions to continuous almost automorphic ones. Further applications of almost automorphic functions in differential and integral equations can be found in  \cite{zaki1974almost,de2012asymptotic,es2016almost,es2017almost,es2020compact,cheban2020periodic,campos2020barycentric,chavez2021almost,choquehuanca2023almost,dads2023massera} and references cited therein.

The notion of piecewise continuous almost periodic/automorphic functions is intrinsically connected with the study of differential equations with piecewise constant argument (DEPCA in short) with almost periodic/automorphic forcing terms \cite{chavez2014discontinuous,chavez2014discon,dads2010pseudo}. A DEPCA, is a differential equation with discontinuous argument of the following type:
	$$y'(t)=g(t,y(t),y([t]))\, ,$$
where $[\cdot]$ is the integer part function. The study of DEPCA was initiated in 1983 with S. M. Shah and J. Wiener's work \cite{busenberg1982models}, in this interesting paper they study DEPCAs as models of "vertically transmitted diseases which are propagated by invertebrate vectors with discrete generations", thus DEPCAs are of importance in understanding diseases dynamics, see also the more recent work \cite{kashkynbayev2020global}. Subsequently, the qualitative theory of DEPCA has been developed intensively, see for instance the works  \cite{cooke1984retarded,aftabizadeh1987oscillatory,cooke2006survey, chiu2019oscillatory,shen2000oscillatory,akhmet2011method} and references therein.



%
%
	Almost periodic type solutions for DEPCA have been  studied, among other works, in  \cite{kupper2002quasi,alonso2000almost,rong1997existence,
	feng2020asymptotically,xia2007existence,dads2010pseudo}. In the almost automorphic setting, it was studied for example in  \cite{van2007almost,dimbour2011almost,
	chavez2014discontinuous,chavez2014discon},
	 see also \cite{ding2017asymptotically} for the asymptotically almost automorphic solutions.

Note that, if $f$ is almost automorphic, then the piecewise continuous function $f([t])$ is not almost automorphic (it is not continuous) but it fits similar properties with $f$ as it was explained in \cite{chavez2014discontinuous}. Since the study of DEPCA with coefficients of the form $f([t])$ for $f$ almost automorphic is different from the classical study of differential equations, it was the main idea in \cite{chavez2014discontinuous,chavez2014discon} to developed the theory of $\mathbb{Z}$-almost automorphic functions with values in the euclidean space in order to study the existence and uniqueness of the almost automorphic solution of the following DEPCAs 
\begin{equation}\label{DEPCA1}
y'(t)=A(t)y(t)+B(t)y([t])+f(t,y(t), y([t]))\, ,
\end{equation}
\begin{equation}\label{DEPCALin1}
     	y'(t)=A(t)y(t)+B(t)y([t])+f(t)\, .
     \end{equation}
The theory of $\mathbb{Z}$-almost automorphic functions (with values in Banach spaces) was exploited in \cite{ding2017asymptotically} to study asymptotically almost automorphic solutions of DEPCAs (\ref{DEPCALin1}) and (\ref{DEPCA1}). Until now, the interesting applications of $\mathbb{Z}$-almost automorphic functions were presented in \cite{chavez2014discontinuous,chavez2014discon,ding2017asymptotically} 
in the context of DEPCAs. The possibility of providing new applications of the $\mathbb{Z}$-almost automorphic functions strongly motivates the present work.

%
%

	In this paper there are two fundamental goals, described as follows:
\begin{itemize}
\item Firstly, for $\mathbb{X}$ a Banach space, we will develop the basic theory of vector valued  $\mathbb{Z}$-almost automorphic functions (and also its asymptotic version), denoted by $\mathbb{Z}AA(\mathbb{R};\mathbb{X})$ ($\mathbb{Z}AA(\mathbb{R}^+;\mathbb{X})$ respectively). It is proved  that, elements in the introduced spaces are always bounded (this is a crucial difference with the works \cite{chavez2014discontinuous,
chavez2014discon,ding2017asymptotically}, where boundedness is part of the definition). Also, it is shown that they are Banach algebras as long as $\mathbb{X}$ is, also a new characterization of compact almost automorphic functions is proposed  (Theorem \ref{charCaa}). Furthermore, composition results and invariance under convolution products of these functional  spaces are presented (section \ref{discontAAf}).
\item Secondly, provide some consequences (and applications) of the functional space $\mathbb{Z}AA(\mathbb{R};\mathbb{X})$.
In this part, using the theory of  $\mathbb{Z}$-almost automorphic functions, a simple proof of the extension Theorem for discrete almost automorphic functions  is presented (Theorem \ref{lastth}), also it is constructed explicit examples of compact almost automorphic functions which are not almost periodic, it is proved that almost automorphic solutions of DEPCA are in fact compact almost automorphic and finally it is studied the $1D$ heat equation in the space $\mathbb{Z}AA(\mathbb{R};\mathbb{R})\cap C(\mathbb{R};\mathbb{R})$. All this happens in section \ref{consequences}.
\end{itemize}
Let us mention that, in \cite{qi2022piecewise} Li and Yuan  have studied piecewise continuous almost automorphic functions but with different interests. In what follows, discontinuous almost automorphic functions mean piecewise continuous almost automorphic functions.

Now, we briefly summarize the organization of the paper. In section \ref{contaaclass} the definition of continuous and discrete almost automorphic functions and some of their basic properties are recalled, in section \ref{discontAAf} it is analysed the theory of $\Z$-almost automorphic functions with values on a general Banach space, and some of its main properties are studied. Finally, in section \ref{consequences} some consequences (and applications) of $\Z$-almost automorphic functions are presented.

\section{Continuous and discrete almost automorphic functions} \label{contaaclass}
In this section, we collect some basic facts about continuous and discrete almost automorphic functions, which will be of interest in subsequent sections.

In what follows, $\X$ and $\Y$ always represent (real or complex) Banach spaces, while $\Z$ represents the group of integer numbers and $\mathbb{R}^+=[0,\infty)$. Also, $B(\mathbb{R};\mathbb{X})$ represents the space of bounded functions from $\mathbb{R}$ to $\mathbb{X}$, and $UC(\mathbb{R};\mathbb{X})$ represents the space of uniformly continuous functions from $\mathbb{R}$ to $\mathbb{X}$, also $\mathbb{F} \in \{\mathbb{R}, \mathbb{C} \}$.

\begin{defi}\label{defAA}
 A continuous function $f:\R \to \X$, is  almost
automorphic if, for every sequence of real numbers $\{s'_n\}$ there
exists a subsequence $\set{s_n}\subseteq \set{s'_n}$ and a function
$\tilde{f}$ such that the following pointwise limits hold:
\begin{equation}\label{Defaa}
\lim_{n\to + \infty}f(t+s_n)=\tilde{f}(t),\quad \lim_{n\to +
\infty}\tilde{f}(t-s_n)=f(t)\, .
\end{equation}
\end{defi}

If, instead of the limits in (\ref{Defaa}), we consider the following limits:
\begin{equation}\label{Defcaa}
\lim_{n\to + \infty}\sup_{t\in K}|| f(t+s_n)-\tilde{f}(t)||=0,\quad \lim_{n\to +
\infty}\sup_{t\in K} ||\tilde{f}(t-s_n)-f(t)||=0\, ,
\end{equation}
for any compact set $K \subset \mathbb{R}$, then $f$ is called {\sl compact almost automorphic}. Also, if the limit $\lim_{n\to + \infty}f(t+s_n)=\tilde{f}(t)$ in   (\ref{Defaa}) is uniform in $\mathbb{R}$, then $f$ is called almost periodic. In this work, we will refer mostly to almost automorphic functions.

We denote the classes of almost automorphic and compact almost automorphic functions by $ AA(\R,\X)$ and $\mathcal{K}AA(\R,\X)$ respectively. 


The following is a useful characterization of compact almost automorphic functions on the real line, see \cite{es2016almost}. For the multidimensional euclidean space, see \cite{CHAKPPINTO2023}.

\begin{thm}\label{CharCAA} 
A function $f:\mathbb{R}\to \mathbb{X}$ is compact almost automorphic if and only if it is almost automorphic and uniformly continuous, thus:
$$\mathcal{K}AA(\mathbb{R}; \mathbb{X})= AA(\mathbb{R}; \mathbb{X})  \cap UC(\mathbb{R}; \mathbb{X})\, .$$
\end{thm}
In the following theorem, we summarize some properties of almost automorphic functions. For the proofs, see, for instance \cite{diaganaBook2013almost,16}.


\begin{thm}
Let $f,g \in AA(\R;\X)$, then
\begin{enumerate}
\item For every $\alpha \in \R$, $f+\alpha g \in AA(\R;\X)$.
\item $AA(\R;\X)$ is a Banach space under the norm of uniform convergence and $\mathcal{K} AA(\R;\X)$ is a closed subspace of $AA(\R;\X)$.
\item $f$ is bounded and if $\tilde{f}$ is the function in Definition $\ref{defAA}$, then
$$||f||_{\infty}= ||\tilde{f}||_{\infty}.$$
\item The range of $f$; i.e. $\mathcal{R}_f=\{f(t)\; :\; t\in \mathbb{R}\}$, is relatively compact in $\X$.
\end{enumerate}
\end{thm}
We remark that, $\mathcal{K} AA(\R;\X)$ and $AA(\R;\X)$ are also unital Banach algebras as long as $\X$ is, see  \cite{CHAKPPINTO2023}.
\begin{ejm}
The classical example of an almost automorphic function which is not almost periodic is the following one: $\psi: \R \to \R$, defined by 
$$\psi(t)=\sin \left( \frac{1}{2+\cos(t)+\cos(\sqrt{2} t)} \right) \; .$$
\end{ejm}
\begin{defi}\label{def2}
 A function $f\in BC(\R\times\X;\Y)$ is said to be almost automorphic (resp. compact almost automorphic) on bounded subsets of $\X$ if, given any bounded subset $\mathcal{B}$ of $\X$ and a sequence $\{s_n'\}$ of real numbers, there exists a subsequence 
$\{s_n\}\subseteq\{s_n'\}$ and a function $\tilde f$, such that the following limits hold:
$$\lim_{n\to\infty}f(t+s_n,x)=\tilde f(t,x)\; ,$$
$$ \lim_{n\to\infty}\tilde f(t-s_n,x)=f(t,x)\; ,$$
where the limits are pointwise in $t\in\R\, $ (resp. uniform for $t$ in compact subsets of $\mathbb{R}$) and uniformly for $x$ in $\mathcal{B}$.
\end{defi}
We denote this class of functions by $AA(\R \times \X, \Y)$ (resp. by $\mathcal{K}AA(\R \times \X, \Y)$ for the compact almost automorphic functions).

Now, Let us define the following function spaces 
$$C_0(\R^+;\X)=\Big{\{}\phi \in C(\R^+;\X): \lim_{t\to +\infty}||\phi(t)||=0 \Big{\}}\; ,$$
$$
C_0(\R^+\times \X;\Y)=\Big{\{} \phi \in C(\R^+\times\X;\Y):\lim_{t\to +\infty}\sup_{x\in \mathcal{B}}||\phi(t,x)||=0,\, \mathcal{B}\,   {\rm a\,  bounded\ 
subset\ of}\ \X  \Big{\}}\; .$$

\begin{defi}\label{DefAAA} A continuous function $f:\R^+\to\X$ (respectively $f:\R^+\times\X\to\Y$) is asymptotically almost
automorphic (respectively asymptotically almost automorphic in $t\in\R$, uniformly on bounded subsets of $\X$) if, 
$f=g+h$, where $g \in AA(\R;\X)$ (respectively $g \in AA(\R\times\X;\Y)$) and $h \in C_0(\R^+;\X) $
(respectively $h \in C_0(\R^+\times \X;\Y)$).
\end{defi}
In definition \ref{DefAAA}, if $g\in \mathcal{K} AA(\R;\X)$ (resp. $g\in \mathcal{K} AA(\R\times \X; \Y)$), then $f$ is called asymptotically compact almost automorphic (resp. asymptotically compact almost automorphic in $t\in\R$, uniformly on bounded subsets of $\X$). 

%
%
We denote by $AAA(\R^+;\X)$ the space of asymptotically almost automorphic functions and by $AAA(\R^+\times\X;\Y)$ the space of functions which are asymptotically almost automorphic in $t\in \R$ uniformly on bounded subsets of $\X$. Also, we denote by $\mathcal{K}AAA(\R^+;\X)$ the space of asymptotically compact almost automorphic functions. Obviously, we have $\mathcal{K}AAA(\R^+;\X) \subset  AAA(\R^+;\X)$.

In the space $AAA(\R^+;\X)$, we can define the following norm
\begin{equation}\label{NormAAA}
||f||:=\sup_{t\in\R}||g(t)|| +\sup_{t\in\R^+}||h(t)||\, ,\, \, f=g+h,\, \, g\in AA(\R;\X), \, \, h\in C_0(\R^+;\X)\, .
\end{equation}

\noindent The following theorem summarizes some properties of asymptotically almost automorphic functions (c.f. \cite{16})
\begin{thm}\label{TeoAAA}
We have
\begin{enumerate}
\item The space $AAA(\R^+;\X)$ becomes a Banach space under the norm $(\ref{NormAAA})$.
\item $\{g(t) : t \in \mathbb{R} \}\subset \overline{ \{f(t) : t \in \mathbb{R} \}} $, where $f=g+h$, with $g\in  AA(\R; \X)$ and $h\in   C_0(\R^+; \X)$.
\item $AAA(\R^+;\X)=AA(\R;\X)\oplus C_0(\R^+;\X)$. That is, the decomposition of an asymptotically almost automorphic function is unique.
\item Let $g \in AAA(\R^+;\X)$, then its range is relatively compact in $\X$.
\end{enumerate}
\end{thm}

\begin{defi}
     	A function $f:\Z \to \mathbb{X}$ is called discrete almost automorphic if, for any sequence $\{s_{n}'\}\subset\Z$, there exists a subsequence $\{s_{n}\}\subset\{s_{n}'\}$  such that the following pointwise limits hold:
     	$$\lim_{n\to +\infty}f(k+s_{n})=:\tilde{f}(k),\;\; \lim_{n\to+\infty}\tilde{f}(k-s_{n})=f(k), \;k\in \Z.$$
     \end{defi}
     
     We denote this space by $\mathbb{AA}(\Z,\mathbb{X})$. A discrete almost automorphic function is also called almost automorphic sequence. We have the following result, see \cite{araya2009almost}
     
\begin{thm}\label{Bseq01}
Let  $f\in \mathbb{AA}(\Z,\mathbb{X})$, then there exists $M>0$ such that $\sup_{n\in \mathbb{Z}}||f(n)|| =M$. Moreover, $\mathbb{AA}(\Z,\mathbb{X})$  is a Banach space under the supremum norm $||\cdot||_{\infty}=\sup_{n\in \mathbb{Z}}||\cdot ||$.
\end{thm}     

\section{$\mathbb{Z}$-almost automorphic type functions}\label{discontAAf}
In this section, we introduce and develop the theory of discontinuous almost automorphic functions  from $\mathbb{R}$ to the Banach space $\mathbb{X}$, which we call $\mathbb{Z}$-almost automorphic functions and also present some of its basic properties.

\subsection{$\mathbb{Z}$-almost automorphic functions}
\begin{defi}
The set $PC_{\Z}(\R; \X)$ consists of functions defined on all $\R$ and which satisfy the following:
\begin{enumerate}
\item $f$ is continuous in the intervals $[n,n+1),\, n \in \Z$.
\item For each $n\in \Z$, the lateral limit $\lim_{t \to n^-}f(t)$ exists, and we define 
 $$f(n^-):=\lim_{t \to n^-}f(t)\, .$$
\end{enumerate}
\end{defi}
Note that, every continuous function belongs to $PC_{\Z}(\R; \X)$, therefore $PC_{\Z}(\R; \X)\not = \emptyset$.

\begin{thm}\label{CompletePC}
The space  $PC_{\Z}(\R; \X)\cap B(\mathbb{R};\mathbb{X})$ is complete under the supremum norm.
\end{thm}

\begin{proof}
Let $\{ f_k \}$ be a Cauchy sequence in  $PC_{\Z}(\R; \X)\cap B(\mathbb{R};\mathbb{X})$. Then, for each $t\in \mathbb{R}$, $\{ f_k(t) \}$ is a Cauchy sequence in $\mathbb{X}$, therefore there exists a function $f:\mathbb{R}\to \X$ such that $f_k(t)\to f(t)$, as $n \to \infty$. Now, pick the compact interval $[n,\theta_n]$ for $n\leq \theta_n < n+1$ and also take $t_m,t_0 \in [n,\theta_n]$ such that $t_m\to t_0$ as $m\to \infty$; then, because of uniform convergence and continuity of $f_k$ in $[n,n+1)$, for each $k$; we have
\begin{eqnarray*}
||f(t_m)-f(t_0) || &\leq & || f_k(t_m)-f(t_m)||+ ||f_k(t_m)-f_k(t_0) ||+  \\
&+& || f_k(t_0)-f(t_0)||\\
&\leq &  \sup_{t\in [n,\theta_n]}|| f(t)-f_k(t)||+ ||f_k(t_m)-f_k(t_0) ||\\
&+& \sup_{t\in [n,\theta_n]}|| f_k(t)-f(t)||\\
&<& \epsilon\, ,
\end{eqnarray*}
which implies that $f(t_m)\to f(t_0)$, as $m\to \infty$. Therefore, $f$ is continuous in each interval $[n,n+1),\, n \in \mathbb{Z}$. Now, we need to prove that $ \lim_{\zeta\to n^-}f(\zeta)$ exists, i.e. $|| \lim_{\zeta\to n^-}f(\zeta)||<\infty$. In fact, by uniform convergence and since $\{f_k\}\subset PC_{\Z}(\R; \X)$, we obtain, for $k_0$ large enough 
\begin{eqnarray*}
|| \lim_{\zeta\to n^-}f(\zeta)|| &\leq &  \lim_{\zeta\to n^-}||f(\zeta)-f_{k_0}(\zeta)|| \\
& + & || f_{k_0}(n^-)|| \\
& <& \infty\, .
\end{eqnarray*}
It follows that, 
$$f(n^{-}):=\lim_{\zeta \to n^{-}}f(\zeta)\, ,$$
is well defined. Therefore, the limit function $f$ belongs to $PC_{\Z}(\R; \X)$ and also is bounded.

\end{proof}
In the following definition, we introduce the class of piecewise continuous almost automorphic functions
\begin{defi}\label{ZAA}
A function $f\in PC_{\Z}(\R; \X)$ is called $\Z$-almost automorphic if, for any sequence of integers $\{s_n'\}\subset \Z$, there exists a subsequence $\{s_n\}\subset \{s_n'\}$ and a function  $\tilde{f}:\R \to \X$ such that the following pointwise limits hold
\begin{equation}\label{limZ}
\lim_{n \to +\infty}f(t+s_n)= \tilde{f}(t),\,\,\, \lim_{n \to +\infty}\tilde{f}(t-s_n)= f(t),\,\,\, t \in \R .
\end{equation}
\end{defi}

Let us denote this class of functions by $\Z AA(\R; \X)$.
\begin{ejm} 
If $f \in AA(\R; \X)$, then $ f([ \cdot ])\in \Z AA(\R; \X)$.
\end{ejm}
In particular, for $\X =\R$ let $\psi : \R \to \R$ be the almost automorphic function defined by $\psi(t)=\sin (\dfrac{1}{2+ \cos(t)+\cos(t\sqrt{2})})$, then the discontinuous function $g: \R \to \R$ defined by 
$$g(t)=\psi([ t ] )=\sin (\dfrac{1}{2+ \cos([ t ])+\cos( [ t ] \sqrt{2})})$$
belongs to $\mathbb{Z} AA(\R; \mathbb{R} )$.
 

\begin{lem} 
$AA(\R; \X) \subsetneq  \Z AA(\R; \X)$.
\end{lem}
\begin{proof}
It is easy to see that, every almost automorphic functions belongs to $\Z AA(\R; \X)$. Now, taking $x_0 \in \X \setminus \{0\} $ and by the previous example, we conclude that the  discontinuous function
$$g(t)=f([t] )=\sin (\dfrac{1}{2+ \cos([t] )+\cos([t] \sqrt{2})}) x_0\, ,$$
belongs to $\Z AA(\R; \X)$, but not to  $AA(\R; \X)$\, .
\end{proof}

The following function is important in the proof of some results

\begin{defi} Let $f \in \Z AA(\R;\X)$. We define the function $F_f:\R \to \X$ in the following way: If $n \leq t \leq n+1$ with $n \in \mathbb{Z}$, then 

\begin{equation}\label{EFE}
F_f(t)= \left\{ \begin{array}{lcc}
             f(t) &   if  & n \leq t < n+1\, , \\
             \\ f((n+1)^-) &  if & t=n+1\, .
             \end{array}
   \right.
\end{equation}
\end{defi}
Note that $F_f$ is uniformly continuous in each interval of the form $[n,n+1]$, for $n \in \mathbb{Z}$.

Observe also that, if $f \in \Z AA(\R; \X)$, then its restriction to $\Z$ is an almost automorphic sequence; therefore, using Theorem \ref{Bseq01} we conclude the following

\begin{lem}\label{lem010}
Let $f \in \Z AA(\R; \X)$, then
$$\sup_{n \in \Z} ||f(n)||<\infty \, .$$
\end{lem}
In the following Lemma, we prove that the jump of discontinuities of a function in  $\Z AA(\R; \X)$ is a bounded set.

\begin{lem}\label{jump}
Let $f \in \Z AA(\R; \X)$, then
$$\sup_{n \in \Z} || f(n)-f(n^-)|| <\infty \, .$$
\end{lem}
\begin{proof}
Firstly, by Lemma \ref{lem010} and since $f(n^{-})$ exists for all $n \in \mathbb{Z}$, then 
$$\sup_{n \in \Z} || f(n)-f(n^-)|| <\infty \, ,$$
is equivalent to 
$$\sup_{n \in \Z: |n|\geq K} || f(n^-)|| <\infty \, ,$$
for an arbitrary large real number $K$, which is equivalent to

$$\sup_{n \in \Z: n\geq K} || f(n^-)|| <\infty \, \,  \, \wedge \ \, \sup_{n \in \Z: n\leq -K} ||f(n^-)|| <\infty\, ,$$
for an arbitrary large real number $K$.

Let us suppose that 
$$\sup_{n \in \Z: n\geq K} || f(n^-)|| =+\infty ;$$
that is, there exists a sequence (strictly increasing) of integer numbers $\{\eta_n'\}\subset \Z$, such that 
$$\lim_{n \to +\infty} || f((\eta_n')^-)||=+\infty \, . $$
Since $f \in \Z AA(\R; \X)$, there exists a subsequence $\{\eta_n \} \subset \{\eta_n'\}$ and a function $\tilde{f}$ such that (\ref{limZ}) holds, and also
\begin{equation}\label{NewEeQq}
 \lim_{n \to +\infty} || f((\eta_n)^-)||=+\infty \, .
\end{equation}

Now, let $\epsilon >0$ be arbitrarily small, then 
\begin{enumerate}
\item By definition of $f((\eta_n)^-)$, there exists a sequence  $\{\delta_n\} \subset (0,1)$ such that if $0<\eta_n -\tau  < \delta_n$, then $||f(\tau)-f(\eta_n^{-})||< \epsilon $. Moreover, the sequence $\{\delta_n\}$ may be chosen in such a way that $\delta_n \to 0$ if $n \to \infty$. That is, there exists $N_0 \in \mathbb{N}$ such that
$$||f(\eta_n^{-})-f(\eta_n-\dfrac{\delta_n}{2})||< \dfrac{\epsilon}{3}, \, \, \forall n \geq N_0\, .$$
Also, there exists $N_1\in \mathbb{N}$ such that,
$$|\delta_n - \delta_m|< \dfrac{\epsilon}{3},\, \, \forall n,m \geq N_1\, .$$

\item There exists $N_2\in \mathbb{N}$ such that, for $m\geq N_2$ (fixed) we have 
$$||f(\eta_n- \dfrac{\delta_m}{2})-\tilde{f}(-\dfrac{\delta_m}{2})|| < \dfrac{\epsilon}{3},\, \, \forall n \geq N_2\, .$$

\item Note also that, (\ref{NewEeQq}) is equivalent to: given $R>0$, there exist $M>0$ and $N_3 \in \mathbb{N}$ such that
\begin{equation}\label{EqFin1}
 ||f(\eta_n^{-})||\geq R,\, \, \, \forall\, \, |\eta_n| \geq M ,\, \, \, \forall  n\geq N_3\, .
\end{equation}

\end{enumerate}
Now, let $T=max \{N_0,N_1,N_2,N_3\}$, $m > T$ fixed and  $R$ be such that 
$$R>1+ ||\tilde{f}(-\dfrac{\delta_m}{2})||\, .$$
Then, if $n_0 > T$ is fixed, and since the function $F_f$ defined in (\ref{EFE}) is uniformly continuous in $[\eta_{n_0}-1, \eta_{n_0}]$, we have
\begin{eqnarray*}
||f(\eta_{n_0}^{-})|| &\leq & ||f(\eta_{n_0}^{-})-f(\eta_{n_0}-\dfrac{\delta_{n_0}}{2})||+ ||f(\eta_{n_0}-\dfrac{\delta_{n_0}}{2})-f(\eta_{n_0}- \dfrac{\delta_m}{2})||+ \\
&+& ||f(\eta_{n_0}- \dfrac{\delta_m}{2})-\tilde{f}(-\dfrac{\delta_m}{2})|| + ||\tilde{f}(-\dfrac{\delta_m}{2})||\\ 
&< & \epsilon + ||\tilde{f}(-\dfrac{\delta_m}{2})||< R \, ,
\end{eqnarray*}
which is a contradiction with (\ref{EqFin1}). Analogously, we proceed for the case 
$$\sup_{n \in \Z: n\leq -M} ||f(n^-)|| = +\infty\, .$$


\end{proof}

\begin{thm}\label{finalthm}
Let $f,g \in \Z AA(\R;\X)$, then, we have
\begin{enumerate}
\item [1)] $f+g \in \Z AA(\R;\X) $
\item [2)] If $\alpha \in \R$, then $\alpha f \in \Z AA(\R;\X)$.
\item [3)] $f$ is bounded.
\item [4)] If $\tilde{f}$ is the function defined in Definition  \ref{ZAA}, then $\tilde{f}$ is bounded, moreover
$$\sup_{t \in \R}||\tilde{f}(t)||=\sup_{t \in \R}||f(t)||\, .$$
\end{enumerate}
\end{thm}
\begin{proof}

\noindent $1)-2)$ Let $\{s_n'\}$ be a sequence of integer numbers. Since $f \in \Z AA(\R;\X)$, there exists a subsequence $\{s_n\} \subset \{s_n'\}$ and a function $\tilde{f}$ such that the following pointwise limits hold
$$\lim_{n \to +\infty}f(t+s_n)= \tilde{f}(t),\,\,\, \lim_{n \to +\infty}\tilde{f}(t-s_n)= f(t),\,\,\, t \in \R .$$

Moreover, since $g \in \Z AA(\R;\X)$, there exists a subsequence $\{r_n\} \subset \{s_n\}$ and a function  $\tilde{g}$ such that the following pointwise limits hold 
$$\lim_{n \to +\infty}g(t+r_n)= \tilde{g}(t),\,\,\, \lim_{n \to +\infty}\tilde{g}(t-r_n)= g(t),\,\,\, t \in \R .$$

Now, if we consider the function  $\tilde{f}+\alpha \tilde{g}$; then, for each $t\in \mathbb{R}$ we have

$$\lim_{n \to +\infty}(f+\alpha g)(t+r_n)= (\tilde{f} + \alpha \tilde{g})(t)\, ,$$
$$ \lim_{n \to +\infty}(\tilde{f} + \alpha \tilde{g})(t-r_n)= (f+\alpha g)(t)\, .$$
Therefore, $f+\alpha g \in \Z AA(\R;\X)$ for every $\alpha \in \R$.

\noindent $3)$ Let us suppose that $f$ is not bounded, then there exists a sequence $\{s_n'\}$ of real numbers such that $\lim_{n\to \infty}||f(s_n')||=\infty$.

On the other hand, there exists a subsequence $ \{s_n\}\subset \{s_n'\}$ such that $s_n=t_n+\xi_n $ with $t_n \in [0,1[ , \,  \lim_{n\to \infty}t_n=t_0\in [0,1]$ and $\xi_n \in \Z$; moreover, $\lim_{n\to \infty}||f(s_n)||=\infty$. Since $f \in \Z AA(\R ;\X)$, the sequence $\xi_n$ can be chosen in such a way that there exists a function $\tilde{f}$ and the pointwise limits in (\ref{limZ}) hold. Now, the proof will be made in two steps:

\begin{enumerate}
\item [a)] {\bf $t_0 =1$}. In this case we have:
\begin{eqnarray*}
||f(s_n)|| &\leq & ||f(t_n+\xi_n)-f(1+\xi_n)||+ ||f(1+\xi_n)||\\
&\leq & ||f(t_n+\xi_n)- f((1+\xi_n)^-)|| + \\ 
&+& || f((1+\xi_n)^-)-f(1+\xi_n)||+\\
&+& ||f(1+\xi_n)-\tilde{f}(1)||+ ||\tilde{f}(1)||\, .
\end{eqnarray*}
Now, from the definition of $f((1+\xi_n)^-)$, the existence of the pointwise limits in (\ref{limZ}) and Lemma \ref{jump}, we have: given $\epsilon >0$, there exists $N_0 \in \mathbb{N}$ such that if $ n \geq N_0$, then 
$$||f(s_n)|| \leq 2\epsilon +||\tilde{f}(1)||+ \sup_{n\in \Z}||f(n^-)-f(n)|| < \infty\, ,$$
a contradiction.

\item [b)] {\bf $t_0 < 1$}. In this case we have :
\begin{eqnarray*}
||f(s_n)|| &\leq & ||f(t_n+\xi_n)-f(t_0+\xi_n)||+ ||f(t_0+\xi_n)||\\
&\leq & ||f(t_n+\xi_n)-f(t_0+\xi_n)||+\\
&+& ||f(t_0+\xi_n)-\tilde{f}(t_0)||+ ||\tilde{f}(t_0)||\, .
\end{eqnarray*}
Since $F_f$ is  uniformly  continuous in $[n,n+1]$ for every $n\in \Z$, and the limits in (\ref{limZ}) hold true, we proceed analogously to the proof of Lemma \ref{jump} in order to have a contradiction.
\end{enumerate}

\noindent $4)$ Note the following inequalities:
$$||\tilde{f}(t)||=\lim_{n \to \infty}||f(t+s_n)|| \leq \sup_{t \in \R} ||f(t)||\, ,$$

$$||f(t)||=\lim_{n \to \infty}||\tilde{f}(t-s_n)|| \leq \sup_{t \in \R} ||\tilde{f} (t)||\, .$$


\end{proof}

\begin{obs}\label{OBS01}
From items 1) and 2) of previous theorem, we conclude that $\Z AA(\R;\X)$ is a vectorial space. Moreover, from item 3) we have that $\Z AA(\R;\X) \subset B(\mathbb{R};\mathbb{X})$, which means that it is a normed vector space with the norm of uniform convergence, i.e. $||\cdot||_{\infty}$\, .
\end{obs}

The following theorem shows that, $\Z AA(\R;\X)$ is complete under the norm of uniform convergence.

\begin{thm}\label{BanachZAA}
$\Z AA(\R;\X)$ is a Banach space under the norm of uniform convergence .
\end{thm}
\begin{proof}
First, from item $3)$ of Theorem \ref{finalthm} (see also Observation \ref{OBS01}), we conclude the inclusion $\Z AA(\R;\X) \subset B(\mathbb{R};\mathbb{X})$. Now, let $\{f_k\}$ be a sequence of functions in $\Z AA(\R;\X)$ that converge uniformly to $f$. Due to Theorem \ref{CompletePC}, we know that $f\in PC_{\Z}(\R; \X)$ and is bounded. Therefore, it only rest to prove that  $f \in \Z AA(\R;\X)$. Let $\{ s_n \}$ be a sequence of integer numbers, then using  the diagonal procedure there exists a subsequence $\{ \eta_n \}$ of $\{ s_n \}$ and a sequence of functions $\{ \tilde{f}_k \}$ such that, for each $t\in \mathbb{R}$, we have
\begin{eqnarray}
\lim_{n\to \infty}f_k(t+\eta_n) &=& \tilde{f}_k(t)\, , \label{ConvDiag1}\\
\lim_{n\to \infty} \tilde{f}_k(t-\eta_n) &=& f_k(t)\, .\label{ConvDiag2}
\end{eqnarray}

Let us prove that for each $t\in \R$,  $\{ \tilde{f_k}(t)\}$ is a Cauchy sequence. Indeed, let us first look at the following inequality:
\begin{eqnarray*}
||\tilde{f}_k(t)-\tilde{f}_l(t)|| &\leq & ||\tilde{f}_k(t)-f_k(t+\eta_n)|| + ||f_k(t+\eta_n)-f_l(t+\eta_n)|| + \\
& \leq & ||f_l(t+\eta_n)-\tilde{f}_l(t)||\,  .
\end{eqnarray*}
Since $\{ f_k \}$ is a Cauchy sequence and by the pointwise  convergence in (\ref{ConvDiag1}) and (\ref{ConvDiag2}), we have that given $\epsilon >0$, there exists $N_0 \in \mathbb{N}$ such that if $n,m \geq N_0$, then
$$||\tilde{f}_n(t)-\tilde{f}_m(t)||< \epsilon\, .$$
Since $\X$ is a Banach space, there exists a function $\tilde{f}$ such that $\tilde{f}_k(t)$ converge to $\tilde{f}(t)$. Now, we show that the following pointwise limits hold
\begin{eqnarray}
\lim_{n \to \infty}f(t+\eta_n)&=&\tilde{f}(t)\, , \label{EeQq1}\\
\lim_{n \to \infty}\tilde{f}(t-\eta_n)&=&f(t)\, .  \label{EeQq2}
\end{eqnarray}
The limit in (\ref{EeQq1}) is obtained from the following inequality 
\begin{eqnarray*}
||f(t+\eta_n)-\tilde{f}(t)|| &\leq & ||f(t+\eta_n)-f_k(t+\eta_n)||+ ||f_k(t+\eta_n)- \tilde{f}_k(t)||+\\
&+& || \tilde{f}_k(t) - \tilde{f}(t)||\, .
\end{eqnarray*}
Analogously, we have the limit in (\ref{EeQq2})\, .
\end{proof}

%

\begin{defi}
A Banach space $\mathbb{X}$, with norm $||\cdot ||_{\mathbb{X}}$ is a Banach algebra, if it is an associative algebra (with addition $+_\mathbb{X}$ and multiplication   $\times_\mathbb{X} $), and  satisfy:
$$||x \times_\mathbb{X}  y||_{\mathbb{X}} \leq ||x||_{\mathbb{X}}\, \, ||y||_{\mathbb{X}}\, , \, \, \,\forall \, x,y \in \mathbb{X}\, .$$
\end{defi}
 
\begin{thm}\label{thm001} Let  $\mathbb{X}$ be a Banach algebra with norm $||\cdot ||_{\mathbb{X}}$,  addition $+_{\mathbb{X}}$ and multiplication $\times_{\mathbb{X}}$, then $\mathbb{Z} AA(\mathbb{R},\mathbb{X})$ is also a Banach algebra with norm $||\cdot ||_{\infty}$ and  the operations: if $f,g \in \mathbb{Z} AA(\mathbb{R},\mathbb{X})$, then
$$ (f+ g)(t):=f(t) +_\mathbb{X} g(t)\, , \, \, \, t \in \mathbb{R}\, ,$$
$$ (f\cdot g)(t):=f(t) \times_\mathbb{X} g(t)\, , \, \, \, t \in \mathbb{R}\, .$$
\end{thm}
\begin{proof}
First of all, note that
\begin{equation}
\notag\Vert f\cdot g\Vert_{\infty}=\sup_{t\in\mathbb{R}}\Vert f(t)\times_{\mathbb{X}} g(t)\Vert_{\mathbb{X}}\leq\left(\sup_{t\in\mathbb{R}}\Vert f(t)\Vert_{\mathbb{X}}\right)\left(\sup_{t\in\mathbb{R}}\Vert g(t)\Vert_{\mathbb{X}}\right)\, ;
\end{equation}
thus :
\begin{equation}
\notag\Vert f\cdot g \Vert_{\infty}\leq\Vert f\Vert_{\infty}\Vert g\Vert_{\infty}\, , \,\, \forall f,g \in \mathbb{Z} AA(\mathbb{R},\mathbb{X})\, .
\end{equation}
Also, since $\mathbb{X}$ is an associative algebra, then  $ \mathbb{Z} AA(\mathbb{R},\mathbb{X})$ is so, and from Theorem \ref{BanachZAA} we have that $\mathbb{Z} AA(\mathbb{R},\mathbb{X})$ is a Banach space under the supremum norm.

Now, let us verify that $\mathbb{Z} AA(\mathbb{R},\mathbb{X})$ is closed under multiplication. For that, let us denote by $+$ and $-$ the operations $+_{\mathbb{X}}$ and $-_{\mathbb{X}}$ (respectively) in $\mathbb{X}$. Then, let $f, h \in \mathbb{Z} AA(\mathbb{R},\mathbb{X})$, and let 
 $\lbrace s^{\prime}_{n}\rbrace\subset\mathbb{Z}$ be any sequence, then there exists a subsequence $\lbrace s_{n}\rbrace\subset\lbrace s^{\prime}_{n}\rbrace$ and there exist functions $g,k :\mathbb{R}\to \mathbb{X}$ such that, for each $t\in\mathbb{R}$, the following limits hold
\begin{equation}
\notag\lim_{n\to \infty}f(t+s_{n})=g(t)\hspace{0.5cm}and\hspace{0.5cm}\lim_{n\to \infty}g(t-s_{n})=f(t)\, ;
\end{equation}
\begin{equation}
\notag\lim_{n\to \infty}h(t+s_{n})=k(t)\hspace{0.5cm}and\hspace{0.5cm}\lim_{n\to \infty}k(t-s_{n})=h(t)\, .
\end{equation}
Let $M,N>0$ be such that $\sup_{t\in\mathbb{R}}\Vert g(t)\Vert_{\mathbb{X}}=M$ and $\sup_{t\in\mathbb{R}}\Vert h(t)\Vert_{\mathbb{X}}=N$. Then:
\begin{eqnarray*}
\Vert (f\cdot h)(t+s_{n})-(g\cdot k)(t)\Vert_{\mathbb{X}}&\leq &\Vert (f\cdot h)(t+s_{n})-g(t)\times_{\mathbb{X}} h(t+s_{n}) \Vert_{\mathbb{X}} \\ 
&+& \Vert g(t)\times_{\mathbb{X}} h(t+s_{n})-(g\cdot k)(t)\Vert_{\mathbb{X}} \\
&\leq & N \Vert f(t+s_{n})-g(t)\Vert_{\mathbb{X}} + M\Vert h(t+s_{n})-k(t)\Vert_{\mathbb{X}}\, .
\end{eqnarray*}

Therefore,
$$
\lim_{n\to \infty} (f\cdot h)(t+s_n) = (g \cdot k)(t)\, .
$$
Similarly, it can be verified that: 
 \begin{equation}
\lim_{n\to \infty}(g\cdot k)(t-s_{n}) = (f\cdot h)(t) \, .
 \end{equation}
Therefore, $\mathbb{Z} AA(\mathbb{R},\mathbb{X})$ is a Banach algebra.
\end{proof}
\begin{rem}\label{Rem001}
 If $\mathbb{X}$ is a unital Banach algebra with identity element $1_{\mathbb{X}}$, then $\mathbb{Z}AA(\mathbb{R},\mathbb{X})$ is also a unital Banach algebra with identity element the constant function equal to $1_{\mathbb{X}}$.
\end{rem}
\begin{cor}\label{Cor001}
Let $P \in \mathbb{F}[X_1,X_2, \cdots , X_n]$ be a polynomial in several variables, Let $\mathbb{X}$ be a Banach algebra and $f_i \in \mathbb{Z} AA(\mathbb{R}; \mathbb{X})$, $\forall i \in \{1,2,\cdots , n\}$. Then $P(f_1,f_2, \cdots , f_n) \in \mathbb{Z} AA(\mathbb{R}; \mathbb{X})$.
\end{cor}

We conclude this section with the following new and useful characterization of compact almost automorphic functions (compare with Theorem \ref{CharCAA})
\begin{thm}\label{charCaa}
A function $f:\mathbb{R}\to \mathbb{X}$ is compact almost automorphic if and only if it is $\mathbb{Z}$-almost automorphic and uniformly continuous, thus:
$$\mathcal{K}AA(\mathbb{R}; \mathbb{X})=\mathbb{Z} AA(\mathbb{R}; \mathbb{X})  \cap UC(\mathbb{R}; \mathbb{X})\, .$$
\end{thm}
\begin{proof}
The  inclusion $\mathcal{K}AA(\mathbb{R}; \mathbb{X}) \subseteq \mathbb{Z} AA(\mathbb{R}; \mathbb{X})  \cap UC(\mathbb{R}; \mathbb{X})\, $ follows from Theorem \ref{CharCAA}. In order to prove the other inclusion and in light of Theorem \ref{CharCAA} again, it is enough to prove that if $f \in \mathbb{Z} AA(\mathbb{R}; \mathbb{X})  \cap UC(\mathbb{R}; \mathbb{X}), $ then $f\in AA(\mathbb{R}; \mathbb{X})$.

Let $\{s'_{n}\}_{n\in\N}$ be an arbitrary sequence of real numbers, then there exists a subsequence $\{s_{n}\}_{n\in\N}\subset\{s'_{n}\}_{n\in\N}$ of the form $s_n=t_n+\xi_n$. with $\xi_n\in\Z$ and $t_n\in [0,1[$ such that $\ds\lim_{n\to\infty}t_{n}=t_{0}\in [0,1]$. Moreover, $\{\xi_{n}\}_{n\in\N}$ can be chosen such that the following pointwise limits hold
\begin{equation}
		\ds\lim_{n\to \infty}f(t+\xi_{n})=:\tilde{f}(t),\quad \lim_{n\to \infty}\tilde{f}(t-\xi_{n})=f(t), \quad t\in\R\, .
		\label{2222}
	\end{equation}
Since $f$ is uniformly continuous, then $\tilde{f}$ is  uniformly continuous too. Now, let us consider the inequalities
	$$||f(t+s_n)-\tilde{f}(t+t_{0})||$$
	$$\leq ||f(t+t_n+\xi_{n})-f(t+t_0+\xi_{n})||+||f(t+t_0+\xi_{n})-\tilde{f}(t+t_{0})||.$$
From this inequality, the uniform continuity of $f$ and the pointwise limits in (\ref{2222}) we have
$$\lim_{n\to \infty}f(t+s_n)=\tilde{f}(t+t_{0})\, .$$
Analogously, we have
$$\lim_{n\to \infty}\tilde{f}(t+t_{0}-s_n)=f(t)\, .$$

%
\end{proof}

This Theorem says that, in order to prove that a function $f$ is compact almost automorphic, it is enough to verify that $f$ is $\mathbb{Z}$-almost automorphic (i.e. the sequences are taken only from the integers) and also that $f$ is uniformly continuous.

\subsection{Asymptotically $\mathbb{Z}$-almost automorphic functions}

\begin{thm}\label{NewTwoZAA}
Let $f \in \Z AA(\mathbb{R}; \X)$ and $\lim_{t \to +\infty} f(t)=0$, then $f$ is the null function.
\end{thm}
\begin{proof}
Let $\{ s_n\}$ be the sequence of natural numbers, i.e. $s_n=n$. Since $f$ is in $\Z AA(\mathbb{R}; \mathbb{X})$, there  exists a subsequence $\{ \sigma_n \}$ of $\{ s_n \}$ and a function $\tilde{f}$ such that the following pointwise limits hold
$$\lim\limits_{n \to +\infty} f(t+\sigma_n)=:\tilde{f}(t)\, ,$$
		$$\lim\limits_{n \to +\infty}\tilde{f}(t-\sigma_n)=f(t)\, .$$
From this, and using the fact that $\lim_{t \to +\infty} f(t)=0$ we conclude that $\tilde{f}$ is the null function, thus $f$ itself is the null function.
\end{proof}

Before to define the class of asymptotically piecewise continuous almost automorphic functions, let us introduce the set $\Z C_0(\R^+; \X)$ as follows

$$\Z C_0(\R^{+}; \X):=\left\lbrace h\in PC_{\mathbb{Z}}(\mathbb{R}^+; \mathbb{X} )\, : \, \lim_{t \to \infty}h(t)=0 \right\rbrace .$$

\begin{defi}
A function $f\in PC_{\mathbb{Z}}(\mathbb{R}^+; \mathbb{X} )$ is called asymptotically $\Z$-almost automorphic, if it admits a decomposition of the form
$$f=g+h\, ,$$
where $g\in \Z AA(\R; \X)$ and $h \in \Z C_0(\R^+; \X)$.
\end{defi}
We denote this class of functions by $\Z AAA(\R^+; \X)$, this class of functions was introduced in \cite{ding2017asymptotically} under the name of $\Z$-asymptotically almost automorphic functions. Note also that, every element in $\Z C_0(\R^+; \X)$ is bounded, then every element in $\Z AAA(\R^+; \X)$ is also  bounded; thus, we can define the supremum norm in the space $\Z AAA(\R^+; \X)$.

The following theorem summarizes interesting properties of the space $\Z AAA(\R^+; \X)$.

\begin{thm}\label{PropiZAA} We have:
\begin{enumerate}
\item $\Z AAA(\R^+;\X)=\Z AA(\R;\X)\oplus \Z C_0(\R^+;\X)$; i.e. the decomposition of every function in  $\Z AAA(\R^+; \X)$ is unique.
\item $\{g(t) : t \in \mathbb{R} \}\subset \overline{ \{f(t) : t \in \mathbb{R}^+ \}} $, where $f=g+h$, with $g\in \Z AA(\R; \X)$ and $h\in \Z C_0(\R^+; \X)$.
\item If $f \in AAA(\R^+; \X)$, then $f([\cdot ])\in \Z AAA(\R^+; \X)$.
\item Let $f=g+h$,  where $g\in \Z AA(\R; \X)$ and $h\in \Z C_0(\R^+; \X)$. If $f$ is uniformly continuous, then $g$ is also uniformly continuous.
\item Let $f\in \mathbb{Z}AAA(\R^+; \X)$ be uniformly continuous, then $f\in \mathcal{K}AAA(\R^+; \X)$.
\item $\Z AAA(\R^+;\X)$ is a Banach space under the supremum norm.
\end{enumerate}
\end{thm}
\begin{proof}
The proof of the first four items are analogous to the ones given in  \cite{ding2017asymptotically}. Item $5$ is a consequence of item $4$ and Theorem \ref{charCaa}. For item $6$, it is enough to see that if $f= g+h$, where $g\in \Z AA(\R; \X)$ and $h\in \Z C_0(\R^+; \X)$, then
$$||g||_{\infty} + ||h||_{\infty}\leq 3 ||f||_{\infty}\, .$$
\end{proof}

We also have the following
\begin{thm}\label{thm0001} Let  $\mathbb{X}$ be a Banach algebra with norm $||\cdot ||_{\mathbb{X}}$,  addition $+_{\mathbb{X}}$ and multiplication $\times_{\mathbb{X}}$; then, $\mathbb{Z} AAA(\mathbb{R}^+,\mathbb{X})$ is also a Banach algebra with norm $||\cdot ||_{\infty}$ and  the operations: if $f,g \in \mathbb{Z} AAA(\mathbb{R}^+,\mathbb{X})$, then
$$ (f+ g)(t):=f(t) +_\mathbb{X} g(t)\, , \, \, \, t \in \mathbb{R}\, ,$$
$$ (f\cdot g)(t):=f(t) \times_\mathbb{X} g(t)\, , \, \, \, t \in \mathbb{R}\, .$$
Moreover,
\begin{itemize}
\item  If $\mathbb{X}$ is a unital Banach algebra with identity element $1_{\mathbb{X}}$, then $\mathbb{Z}AAA(\mathbb{R}^+,\mathbb{X})$ is also a unital Banach algebra with identity element the constant function equal to $1_{\mathbb{X}}$.

\item If $P \in \mathbb{F}[X_1,X_2, \cdots , X_n]$ is a polynomial in several variables and $f_i \in \mathbb{Z} AAA(\mathbb{R}^+; \mathbb{X})$, $\forall i \in \{1,2,\cdots , n\}$. Then $P(f_1,f_2, \cdots , f_n) \in \mathbb{Z} AAA(\mathbb{R}^+; \mathbb{X})$.
\end{itemize}
\end{thm}

\subsection{Composition results}
The results in this subsection are inspired by  (and generalize) the ones given in \cite{ding2008asymptotically}, their proofs are very similar.

Let $\mathbb{T}$ be either $\mathbb{R}^+$ or $\mathbb{R}$; then, we denote by  $PC_{\mathbb{Z}}^1(\mathbb{T} \times \mathbb{X}; \mathbb{Y})$  the set of functions $h:\mathbb{T} \times \mathbb{X} \to  \mathbb{Y}$  such that $h(\cdot,\cdot)$ is continuous in $\mathbb{X}$ (i.e. in the second variable) and for each $x\in \mathbb{X}$, $h_x(\cdot):=h(\cdot,x)\in  PC_{\mathbb{Z}}(\mathbb{T}; \mathbb{Y})$.

\begin{defi}\label{DefZAA}
 A function $f\in PC_{\mathbb{Z}}^1(\R\times\X;\Y)$ is said to be $\mathbb{Z}$-almost automorphic on bounded subsets of $\X$ if, given any bounded subset $\mathcal{B}$ of $\X$ and a sequence $\{s_n'\}\subset \mathbb{Z}$, there exists a subsequence 
$\{s_n\}\subseteq\{s_n'\}$ and a function $\tilde f$, such that the following limits hold:
$$\lim_{n\to\infty}f(t+s_n,x)=\tilde f(t,x)\; ,$$
$$ \lim_{n\to\infty}\tilde f(t-s_n,x)=f(t,x)\; ,$$
where the limits are pointwise in $t\in\R\, $ and uniformly for $x$ in $\mathcal{B}$.
\end{defi}
We denote this class of functions by $\mathbb{Z} AA(\R \times \X, \Y)$.\\
%
Now, let us define the following set
\begin{eqnarray*}
\mathbb{Z}_0C(\mathbb{R}^+\times \mathbb{X}; \mathbb{Y})&:=&
\Big{\{} \phi \in PC_{\mathbb{Z}}^1(\mathbb{R}^+\times \mathbb{X}; \mathbb{Y} ):\lim_{t\to +\infty}\sup_{x\in \mathcal{B}}||\phi(t,x)||=0, \\
&\, & \,  \, \mathcal{B}\,   {\rm a\,  bounded\ 
subset\ of}\ \X  \Big{\}}.
\end{eqnarray*}

\begin{defi}\label{DefZAAA} 
A function $f:\R^+\times\X\to\Y$  is asymptotically $\mathbb{Z}$-almost automorphic in $t\in\R$, uniformly on bounded subsets of $\X$ if, 
$f=g+h$, where  $g \in \mathbb{Z} AA(\R \times \X, \Y)$ and $h \in \mathbb{Z}_0C(\mathbb{R}^+\times \mathbb{X}; \mathbb{Y})$.
\end{defi}
We denote this class of discontinuous almost automorphic functions by $\Z AAA(\mathbb{R}^+\times \mathbb{X}; \mathbb{Y})$.

\begin{defi}

Let $\mathcal{B}\subset \mathbb{X}$ and $\mathcal{T} \subset \mathbb{R}$. We denote by $\mathcal{C}_{\mathcal{B}}(\mathcal{T}\times \mathbb{X}; \mathbb{Y})$, the set of all functions $f:\mathcal{T}\times \mathbb{X} \to \mathbb{Y}$ satisfying $f(t,\cdot )$ is uniformly continuous in $\mathcal{B}$ uniformly for $t\in \mathcal{T}$; thus, $\forall \epsilon >0, \exists \delta >0$ such that if $x_1,x_2 \in \mathcal{B}$ with $||x_1-x_2||< \delta$, then $||f(t,x_1)-f(t,x_2)||<\epsilon $, $\forall t \in \mathcal{T}$.
\end{defi}
\begin{ejm}
Let $b :\mathbb{R}\to \mathbb{R}$ be a bounded function and define $f:\mathbb{R}\times \mathbb{X}\to \mathbb{R}$ by $f(t,x)=b([t]) \sin(||x||)$, then  $f\in \mathcal{C}_{\mathbb{X}}(\mathbb{R}\times \mathbb{X}; \mathbb{R})$.
\end{ejm}

Let us note that, if $f \in \Z AA(\mathbb{R}\times \mathbb{X}; \mathbb{Y}) \cap  \mathcal{C}_{\mathcal{B}}(\mathbb{R}\times \mathbb{X}; \mathbb{Y})$, then the function $\tilde{f}$ defined in Definition \ref{DefZAA} also belongs to $\mathcal{C}_{\mathcal{B}}(\mathbb{R}\times \mathbb{X}; \mathbb{Y})$. With this clarification, we can state the following results

\begin{thm}\label{ThmCompAA} Let $x\in \Z AA(\mathbb{R}; \mathbb{X})$ and $\mathcal{B}=\{ x(t): t \in \mathbb{R} \}$. If $f \in \Z AA(\mathbb{R}\times \mathbb{X}; \mathbb{Y}) \cap  \mathcal{C}_{\mathcal{B}}(\mathbb{R}\times \mathbb{X}; \mathbb{Y})$, then  $f(\cdot, x(\cdot)) \in \Z AA(\mathbb{R}; \mathbb{Y})$.
\end{thm}

\begin{lem} Let $\mathcal{B}$ be a bounded subset of $\mathbb{X} $ and $f\in \Z AA(\mathbb{R}\times \mathbb{X}; \mathbb{Y})\cap  \mathcal{C}_{\mathcal{B}}(\mathbb{R}^+\times \mathbb{X}; \mathbb{Y})$, then $f \in \mathcal{C}_{\mathcal{B}}(\mathbb{R}\times \mathbb{X}; \mathbb{Y})$.
\end{lem}

\begin{thm}
Let $x \in \mathbb{Z} AAA(\mathbb{R}^+;\mathbb{X})$, $\mathcal{B}=\{ x(t), t \in \mathbb{R}^+\}$ be relatively compact in $\mathbb{X} $. If $f=f_a+h_0 \in  \Z AAA(\mathbb{R}^+\times \mathbb{X}; \mathbb{Y})\cap  \mathcal{C}_{\mathcal{B}}(\mathbb{R}^+\times \mathbb{X}; \mathbb{Y})$, where $f_a \in  \Z AA(\mathbb{R}\times \mathbb{X}; \mathbb{Y})$  and  $f_0\in C_0(\R^+\times \X;\Y)$, then $f(\cdot, x(\cdot))\in  \mathbb{Z} AAA(\mathbb{R}^+;\mathbb{Y})$.
\end{thm}
%

\begin{thm}
Let $f \in \Z AAA(\mathbb{R}^+\times \mathbb{X}; \mathbb{Y})\cap  \mathcal{C}_{\mathcal{\mathbb{X}}}(\mathbb{R}^+\times \mathbb{X}; \mathbb{Y})$  and let $x\in \Z AAA(\mathbb{R}^+; \mathbb{X})$, then $f(\cdot, x(\cdot)) \in \Z AAA(\mathbb{R}^+; \mathbb{Y})$.
\end{thm}
In particular, we have the following corollary
\begin{cor}
Let $f \in \Z AAA(\mathbb{R}^+\times \mathbb{X}; \mathbb{Y})$ be Lipschitz in the second variable and let $x\in \Z AAA(\mathbb{R}^+; \mathbb{X})$, then $f(\cdot, x(\cdot)) \in \Z AAA(\mathbb{R}^+; \mathbb{Y})$.
\end{cor}

%

\subsection{$\mathbb{Z}$-almost automorphic type functions and convolution products}
In this subsection, we prove invariance of the introduced discontinuous $\mathbb{Z}$-almost automorphic type functions under the action of convolution products. We start with the following Proposition

%
\begin{prop}\label{THF1}
Let $\Phi\in L^1(\R ;\R)$, then the operator $L$ defined by
$$(L f)(t)=\int_{-\infty}^{\infty}\Phi(t-s)f(s)ds =\Phi * f(t)\, ,$$
leaves invariant the space $\Z AA(\R;\X)$.
\end{prop}
\begin{proof}
Let $f\in \Z AA(\R;\X)$.

\begin{enumerate}
\item First, let us prove that $\Phi * f$ belongs to the space $PC_{\Z}(\R; \X)$. Note that, for $t_0\in \mathbb{R}$ we have
\begin{eqnarray}
\Phi * f(t)-\Phi * f(t_0)&=& \int_{\mathbb{R}}\Phi(\eta)\big{(} f(t-\eta)-f(t_0-\eta) \big{)}d\eta \nonumber \\
&=& \sum_{k\in \mathbb{Z}}\int_k^{k+1}\Phi(\eta)\big{(} f(t-\eta)-f(t_0-\eta) \big{)}d\eta\, . \label{NeWEeQq}
\end{eqnarray}

\noindent $a)$ For $t_0=n \in \mathbb{Z}$ and a sequence $\{t_m \} \subset [n,n+1)$ such that $t_m\to n^{+}$ when $m\to \infty$, we have
\begin{eqnarray*}
\int_k^{k+1}\Phi(\zeta)\big{(} f(t_m-\zeta)-f(n-\zeta) \big{)}d\zeta &=& \int_{n-k-1}^{n-k}\Phi(n-\zeta)\big{(} f(t_m-n+\zeta)-f(\zeta) \big{)}d\eta\, .
\end{eqnarray*}
Then, since we have the pointwise limit $\lim_{m\to \infty}\big{(} f(t_m-n+\zeta)-f(\zeta)\big{)} = 0$ in $[n-k-1,n-k)$, $f$ is bounded and $\Phi \in L^1(\mathbb{R},\mathbb{R})$, we conclude
$$\lim_{m\to \infty}\int_{n-k-1}^{n-k}\Phi(n-\zeta)\big{(} f(t_m-n+\zeta)-f(\zeta) \big{)}d\eta=0\, .$$

\noindent $b)$ For $t_0\in (n, n+1),\, n \in \mathbb{Z}$. Pick a sequence $\{t_m \} \subset (n,n+1)$ such that $t_m\to t_0$ when $m\to \infty$. Also, for $\delta >0$ arbitrarily small, denote by $I_{\delta}(n,k)$ the open interval $I_{\delta}(n,k) := (n-k-\delta, n-k+\delta)$ and by $\bigtriangleup f(t_m,t_0,\zeta)$ the difference  $\bigtriangleup f(t_m,t_0,\zeta):= f(t_m-t_0+\zeta)-f(\zeta)$, then
\begin{eqnarray*}
\int_k^{k+1}\Phi(\zeta)\big{(} f(t_m-\zeta)-f(t_0-\zeta) \big{)}d\zeta &=& 
\int_{t_0-k-1}^{t_0-k}\Phi(t_0-\zeta) \bigtriangleup f(t_m,t_0,\zeta) d\zeta \\
 &=& \int_{t_0-k-1}^{n-k-\delta}\Phi(t_0-\zeta) \bigtriangleup f(t_m,t_0,\zeta) d\zeta + \\
&+& \int_{n-k+\delta}^{t_0+k}\Phi(t_0-\zeta) \bigtriangleup f(t_m,t_0,\zeta) d\zeta +\\
&+&  \int_{I_{\delta}(n,k)}\Phi(t_0-\zeta) \bigtriangleup f(t_m,t_0,\zeta) d\zeta \, .\\
\end{eqnarray*}
From this decomposition, we have (analogously to the step $a)$)
$$\lim_{m\to \infty}\int_{t_0-k-1}^{n-k-\delta}\Phi(t_0-\zeta) \bigtriangleup f(t_m,t_0,\zeta) d\zeta =0\, ,$$

$$\lim_{m\to \infty} \int_{n-k+\delta}^{t_0+k}\Phi(t_0-\zeta) \bigtriangleup f(t_m,t_0,\zeta) d\zeta =0\, ,$$
and, since 
\begin{eqnarray*}
\big{|}\int_{I_{\delta}(n,k)}\Phi(t_0-\zeta) \bigtriangleup f(t_m,t_0,\zeta) d\zeta\big{|} &\leq &  2||f||_{\infty} \int_{\mathbb{R}}|\Phi(t_0-\zeta)|\mathds{1}_{I_{\delta}(m,k)}(\zeta) d\zeta\, ;
\end{eqnarray*} 
where, $\mathds{1}_{I_{\delta}(m,k)}(\cdot)$ is the indicator function of $I_{\delta}(m,k)$, then $$\lim_{\delta \to 0} \int_{I_{\delta}(n,k)}\Phi(t_0-\zeta) \bigtriangleup f(t_m,t_0,\zeta) d\zeta=0\, .$$

Therefore, from $a)$, $b)$ and (\ref{NeWEeQq}) we conclude that the convolution $\Phi * f$ is continuous in each interval $[n,n+1),\, \, n \in \mathbb{Z}$. Also, since $\Phi * f$ is bounded, the limit $\lim_{\eta \to n^{-}}\Phi * f(\eta)$ exists, therefore 

$$\Phi * f(n^{-}):=\lim_{\eta \to n^{-}}\Phi * f(\eta)\, ,$$
is well defined. Thus, $\Phi * f\in PC_{\Z}(\R; \X)$.

%
%
%

\item Let $\{s_n'\}$ be a sequence in $\Z$; then, there exists a subsequence $\{s_n\}\subset \{s_n'\}$ and a function $\tilde{f}$ such that the following pointwise limits hold
\begin{equation*}\label{EEqq001}
\lim_{n\to \infty}f(t+\eta_n)=\tilde{f}(t)\, ,\, \, \, 
\lim_{n\to \infty}\tilde{f}(t-\eta_n)=f(t)\, .
\end{equation*}
Now, defining the following function
$$g(t):=\int_{-\infty}^{\infty}\Phi(t-s)\tilde{f}(s)ds\, ,$$
we have 
$$Lf(t+s_n)-g(t)=\int_{-\infty}^{\infty}\Phi(\zeta) \Big{(}f (t-\zeta + s_n)-\tilde{f}(t-\zeta) \Big{)} d\zeta\, .$$
Therefore, using the Lebesgue's Dominated Convergence Theorem, we conclude the following pointwise limit
$$\lim_{n\to \infty}Lf(t+s_n)=g(t)\, .$$
Analogously, we obtain
$$\lim_{n\to \infty}g(t-s_n)=Lf(t)\, .$$
\end{enumerate}

\end{proof}
%
The proofs of the following propositions are omitted, because they are deduced from the previous one and also because they are very similar to the almost automorphic situation; see for instance  \cite{CHAKPPINTO2023BBMSS,chavez2021almost,ding2008asymptotically}.
\begin{prop}\label{THF2}
Let $\Phi \in L^1([0,\infty);\R)$, then the operator $L$ defined by 
$$(L f)(t)=\int_{-\infty}^{t}\Phi(t-s)f(s)ds\, ,$$
leaves invariant the space $\Z AA(\R;\X)$.
\end{prop}

\begin{prop}\label{THF201}
Let $\Phi \in L^1([0,\infty);\R)$, such that
\begin{enumerate}
\item $$\lim_{t \to \infty}\int_t^{+\infty}|\Phi(s)|ds =0\, .$$
\item There exists $r_0>0$ such that for every $r>r_0$ we have
$$\lim_{t \to \infty}\int_0^{r}|\Phi(t-s)|ds =0\, .$$
\end{enumerate} 
Then, the operator $L$ defined by 
$$(L f)(t)=\int_{0}^{t}\Phi(t-s)f(s)ds\, ,$$
leaves invariant the space $\Z AAA(\R^+;\X)$.
\end{prop}

In particular, we have (see \cite[Lemma 2.6]{ding2008asymptotically} for the case $ AAA(\R^{+};\X)$)
\begin{cor} Let $R(t)_{t\geq 0}$ be a family of continuous linear operators on $\X$ such that $||R(t)||\leq M e^{-\omega t} $ for all $t\geq 0$ and some fixed positive constants $M, \omega$. Then, the operator $L$ defined by 
$$(L f)(t)=\int_{0}^{t}R(t-s)f(s)ds\, ,\, \, \, t\geq 0\, ,$$
leaves invariant the space $\Z AAA(\R^+;\X)$.
\end{cor}

\section{Some consequences}\label{consequences}


\subsection{Extension of almost automorphic sequences to compact almost automorphic functions: A simple proof}

A classical problem within the theory of almost automorphic functions is the following: given an almost automorphic sequence $S:\Z \to \X$,  is there an almost automorphic function $f:\R \to \X$ such that $f|_{\Z}=S$?. This problem is known as the extension problem of almost  automorphic sequences to almost automorphic functions defined in $\R$ and was answered positively by Cuevas et. al. \cite{cuevas2012existence}, Fink \cite{fink1969extensions} and R. Yuan \cite{yuan2010favard}.

In this subsection, using the theory of $\Z$-almost automorphic functions we give a very simple proof of the extension Theorem for almost automorphic sequences. The Theorem can be viewed as a  characterization result of almost automorphic sequences.

\begin{thm}\label{lastth}
Let $S:\Z \to \X$ be an almost automorphic sequence, then there exists a compact almost automorphic function $f:\R \to \X$ such that $f|_{\Z}=S $\, .
\end{thm}
\begin{proof}
Let us define the following function
\begin{equation}
f(t):=S(k)+(t-k)(S(k+1)-S(k)) ,\, \, \, k\leq t < k+1 \, .
\end{equation}
We claim that $f\in \Z AA(\R; \X)$. Indeed, let $\{s_n\}$ be a sequence of integer numbers. Then, there exists a subsequence $\{\eta_n\}$ of $\{s_n\}$ and a discrete function $\tilde{S}$ such that the following pointwise limits are satisfied

\begin{equation}\label{EEqq01}
\lim_{n\to \infty}S(k+\eta_n)=\tilde{S}(k)\, ,
\end{equation}

\begin{equation}\label{EEqq02}
\lim_{n\to \infty}\tilde{S}(k-\eta_n)=S(k)\, .
\end{equation}
Now, if we define the function  $\tilde{f}(t)=\tilde{S}(k)+(t-k)(\tilde{S}(k+1)-\tilde{S}(k))$ for $k\leq t <k+1$, then we must show that $f(t+\eta_n)\to \tilde{f}(t)$ and that $\tilde{f}(t-\eta_n)\to f(t)$. 
First, let us observe that if $k\leq t < k+1$, then $k+\eta_n \leq t +\eta_n < k+1 + \eta_n$, then
$$f(t+\eta_n)=S(k+\eta_n)+(t+\eta_n-(k+\eta_n))(S(k+1 +\eta_n)-S(k+\eta_n))$$
$$=S(k+\eta_n)+(t-k)(S(k+1 +\eta_n)-S(k+\eta_n))\, ,$$
from this equality and by the limit in (\ref{EEqq01}) we have
$$\lim_{n \to \infty}f(t+\eta_n)=\tilde{f}(t)\, .$$
Analogously, if $k\leq t < k+1$, then $k-\eta_n \leq t -\eta_n < k+1 - \eta_n$; then
$$\tilde{f}(t-\eta_n)=S(k-\eta_n)+(t-\eta_n-(k-\eta_n))(S(k+1 -\eta_n)-S(k-\eta_n))$$
$$=S(k-\eta_n)+(t-k)(S(k+1 -\eta_n)-S(k-\eta_n))\, .$$
Therefore, using the limit in (\ref{EEqq02}) we obtain 

$$\lim_{n \to \infty}\tilde{f}(t-\eta_n)=f(t)\, .$$

Now, let us prove that $f$ is uniformly continuous. Since $S$ is bounded, let $M>\sup_{n \in \Z}||S(n)||$ with $M>1$; then, given $\epsilon >0$ there  exists $\delta < \dfrac{\epsilon}{4M}$ such that if $|t_1-t_2|<\delta$, then,
\begin{enumerate}
\item If $t_1,t_2 \in [k,k+1[$, we have
\begin{eqnarray*}
||f(t_1)-f(t_2)||&=&||(t_1-t_2)(S(k+1)-S(k))||\\
&\leq & |t_1-t_2|2M\\
& < & \dfrac{\epsilon}{2M} < \epsilon\, .
\end{eqnarray*}
\item If $k<t_1<k+1<t_2<k+2$, then $|t_2-t_1|<\delta$ implies $|t_2-(k+1)|<\delta$; and, since
$$f(t_1)-f(t_2)=(t_1-t_2)(S(k+1)-S(k))-(t_2-(k+1))(S(k+2)-S(k))\, ,$$
we have
\begin{eqnarray*}
||f(t_1)-f(t_2)||&=&|t_1-t_2|||(S(k+1)-S(k))||\\
&+ & |t_2-(k+1)||S(k+2)-S(k)||\\
& < & \delta 4M < \epsilon\, .
\end{eqnarray*}
\end{enumerate}

\end{proof}
The reader may compare the proof of this Theorem with those given for \cite[Theorem 2.1]{cuevas2012existence}, \cite[Theorem 1]{fink1969extensions} and \cite[Theorem 2.8]{yuan2010favard}.

In addition to the previous Theorem of continuous extension, we can also provide the following discontinuous extension Theorem:

\begin{thm}
Let $S:\Z \to \X$ be an almost automorphic sequence, then there exists a function  $f\in \Z AA(\R; \X)$ such that $f|_{\Z}=S$\, .
\end{thm}
\begin{proof}
We define the function $f:\R \to \X$ through $f(t)=S([t])$.
\end{proof}


\subsection{Examples of compact almost automorphic functions which are not almost periodic}
Theorem \ref{lastth} motivates the following examples
\begin{ejm}
Let us consider the almost automorphic sequence $\{x(n)\}$; then, its linear extension $f$ given by Theorem \ref{lastth} is compact almost automorphic and is not almost periodic.
\end{ejm}

\begin{ejm}
Let us consider the almost automorphic sequence $\{x(n)\}$, and the almost automorphic function $f:\mathbb{R}\to \mathbb{R}$. Let $F:\mathbb{R}\to \mathbb{R}$ be the continuous and piecewise smooth function defined, for $n \in \mathbb{Z}$, by

$$
F(t)= \left\{ \begin{array}{lcc} x(n)+2(t-n)\big{(} f(n+1/2)-x(n) \big{)} & if & t \in [n,n+1/2) \\ \\ f(n+1/2)+2(t-n-1/2)\big{(} x(n+1)-f(n+1/2)  \big{)} & if & t \in [n+1/2,n+1) . \\
   \end{array} \right.
$$
Then, $F$ is uniformly continuous and also belongs to the space $\mathbb{Z} AA(\mathbb{R}; \mathbb{R})$. It follows from Theorem \ref{charCaa} that $F$ is compact almost automorphic, moreover it is not almost periodic. Note also that, besides of Theorem \ref{lastth}, the function $F$ in the present example give us another continuous (compact) almost automorphic extension of the discrete almost automorphic function $\{x(n)\}$.
\end{ejm}

\subsection{Compact almost automorphic solutions of DEPCA}

For $p\in \{1,2, \cdots \}$, let us consider the following DEPCAs on $\mathbb{C}^p$ 

\begin{equation}\label{DEPCA}
y'(t)=A(t)y(t)+B(t)y([t])+f(t,y(t), y([t]))\, ,
\end{equation}

\begin{equation}\label{DEPCALin}
     	y'(t)=A(t)y(t)+B(t)y([t])+f(t);
     \end{equation}
Equations (\ref{DEPCA})  and (\ref{DEPCALin}) were studied by A. Ch\'avez and coautors in \cite{chavez2014discon} in the space of almost automorphic functions and by H.S. Ding and co-authors in \cite{ding2017asymptotically} in the space of asymptotically almost automorphic functions, in both works, the principal ingredients are the notions of exponential dichotomy and Bi-almost automorphic functions. We recommend referring to works \cite{chavez2014discon,ding2017asymptotically} for an exploration of these important concepts.

Taking into account the following more general nonlinear differential equation with discontinuous argument
\begin{equation}\label{GenDEPCA}
y'(t)=g(t,y(t),y([t]))\, ,
\end{equation}
we are able to define what we understood by a solution of a DEPCA
	\begin{defi}
		 A function $y(\cdot)$ is a solution of the DEPCA (\ref{GenDEPCA}) in the interval $I$, if it satisfies the following conditions:
 	 \begin{enumerate}
		 	\item $y(\cdot)$ is continuous in all $I$.
		 	\item $y(\cdot)$ is differentiable in all $I$, except possibly in the points $n\in I\cap \Z$  where there should be a lateral derivative.
		 	\item $y(\cdot)$ satisfies the equation in all the intervals $(n,n+1), \; n\in \Z$ as well as is satisfies by the right side derivative in each $n\in \Z$.
		 \end{enumerate}
	\end{defi}

The following proposition is also presented in \cite[Lemma 4.1]{chavez2014discontinuous}. Here, we clarified its proof.
\begin{prop}\label{propUC}
Let $A(\cdot), B(\cdot)$ and $f(\cdot)$ be bounded and locally integrable functions. Then, every bounded solution of DEPCA  (\ref{DEPCALin}) is uniformly continuous.
\end{prop}
\begin{proof}
Let $y$ be a bounded solution of DEPCA (\ref{DEPCALin}); then, there exists $M>0$ such that
	$$M=\ds\sup_{t\in \R}\big|A(t)y(t)+B(t)y([t])+f(t)\big|\, .$$
Also, $\forall t\in I_{k}=(k,k+1)$ with $\; k\in \Z$, we have $|y'(t)|\leq M.$ Now, let $\varepsilon>0$, there exists $ \delta=\dfrac{\varepsilon}{2M}>0$ such that: if $t_{1},t_{2}\in I_{k}$, with $|t_{1}-t_{2}|<\dfrac{\varepsilon}{2M},\; $ then
	
	\begin{align}
		|y(t_{1})-y(t_{2})|&=|y'(\xi)(t_{1}-t_{2})| 
		<\varepsilon.
	\end{align}
On the other hand, if $t_{1}<k<t_{2},$ with $k\in\Z$ we have (using the  continuity of $y$) 
	
	\begin{align}
|y(t_{1})-y(t_{2})|&\leq|y(t_{1})-y( {k})|+|y( {k})-y(t_{2})|\nonumber\\
		&\leq \ds\lim_{\eta\to k^{-}}|y(t_{1})-y( {\eta})|+\ds\lim_{\varphi\to k^{+}}|y( {\varphi})-y(t_{2})|\nonumber\\
		&\leq \ds\lim_{\eta\to k^{-}}|\ds\int_{t_{1}}^{\eta}y'(t)dt|+\ds\lim_{\varphi\to k^{+}}|\int_{\varphi}^{t_{2}}y'(t)dt|\nonumber\\
		&\leq \ds\lim_{\eta\to k^{-}}\ds\int_{t_{1}}^{\eta}|y'(t)dt|+\ds\lim_{\varphi\to k^{+}}\int_{\varphi}^{t_{2}}|y'(t)dt|\nonumber\\
		&\leq M(k-t_{1})+M(t_{2}-k)<\varepsilon\, .
	\end{align}
\end{proof}


Let $y$ be a solution of (\ref{DEPCALin}) and $\Phi(t)$ a fundamental matrix solution of the system $y'(t)=A(t)y(t)$. Then, using the variation of constant formula, we observe that the solution $y$  for  $ t\in [n;n+1[,\, n \in \Z$, satisfies:
   
     \begin{equation}\label{2222Int}
     	y(t)=\bigg(\Phi(t,n)+\int_{n}^{t}\Phi(t,u) B(u)du \bigg)y(n)+\int_{n}^{t}\Phi(t,u)f(u)du,
     \end{equation}
     
     where $\Phi(t,s)=\Phi(t)\Phi^{-1}(s)$. Since $y$ is continuous, taking the limit  $t\to (n+1)^{-}$ in (\ref{2222Int}), we arrive to  the difference equation     
     \begin{equation}
        y(n+1)=C(n)y(n)+h(n),\; n\in\Z,
     \end{equation}
where
\begin{equation}\label{Coeffn}
C(n)=\Phi(n+1,n)+\int_{n}^{n+1}\Phi(n+1,u)B(u)du,\, \,  h(n)=\int_{n}^{n+1}\Phi(n+1,u)f(u)du.
\end{equation}

As it was mentioned in \cite{chavez2014discon}, a solution $y = y(t)$ of the DEPCA (\ref{DEPCALin}) is defined on $\R$ if and only if the matrix
	\begin{equation}
		I+\ds\int_{\tau}^{t}\Phi(\tau,u)B(u)du,
	\end{equation}
is invertible for $t,\tau\in[n,n+1], n\in \Z$, where $I$ is the identity matrix.
%
%
%
%
%
%
%
%
%
%
%

Now, using Proposition \ref{propUC} and item $5$ in Theorem \ref{PropiZAA}, we improve \cite[Theorem 3.6]{chavez2014discon} as follows 
\begin{thm} Let $A(t), B(t)$ be almost automorphic matrices, $f\in \mathbb{Z} AA(\mathbb{R}; \mathbb{C}^p) $ and suppose that the system $y(n+1)=C(n)y(n), n \in \mathbb{Z}$, where $C(n)$ is given in  (\ref{Coeffn}) has a Bi-almost automorphic exponential dichotomy with parameters $(\alpha, K,P)$. Then equation (\ref{DEPCALin}) has a unique compact  almost automorphic solution.
\end{thm}

\noindent We say that $f:\mathbb{R}\times \mathbb{C}^p \times \mathbb{C}^p \to \mathbb{C}^p$ is $L_f$-Lipschitz, if there exists a positive constant $L_f>0$, such that
$$|f(t,x,y)-f(t,z,w)| \leq L_f(|x-y|+|z-w|), \forall t \in \mathbb{R},\,  \forall\, \,  (x, y),(z, w) \in \mathbb{C}^p \times \mathbb{C}^p\, . $$
Now, we improve \cite[Theorem 3.8]{chavez2014discon} as follows

\begin{thm} Let $A(t), B(t)$ be almost automorphic matrices, $f\in \mathbb{Z} AA(\mathbb{R} \times \mathbb{C}^p \times \mathbb{C}^p; \mathbb{C}^p) $ and $L_f$-Lipschitz; moreover suppose that the system $y(n+1)=C(n)y(n), n \in \mathbb{Z}$, where $C(n)$ is given in  (\ref{Coeffn}) has a Bi-almost automorphic exponential dichotomy with parameters $(\alpha, K,P)$. Then, there exists a constant $M^* >0$ such that, if $0<L_f<M^*$, then equation (\ref{DEPCA}) has a unique compact  almost automorphic solution.
\end{thm}
Analogously, the conclusion of \cite[Theorem 4.2]{ding2017asymptotically} can be improved in such a way that the solution be asymptotically compact almost automorphic.

Finally, if we consider the following model with piecewise constant argument:
\begin{equation}\label{EeQqmodel}
y'(t)=-\delta(t)y(t)+p(t)f(y([t]))\, ,
\end{equation}
where $\delta(\cdot)$ and $ p(\cdot)$ are positive almost automorphic functions, $0<\delta_{-}=\inf_{s\in \mathbb{R}}\delta(s)$ and $f$ is a positive $\gamma$-Lipschitz function, then \cite[Theorem 4.2]{chavez2014discon} becomes

\begin{thm} In the above conditions, for $\gamma$ sufficiently small, equation (\ref{EeQqmodel}) has a unique compact almost
automorphic solution.
\end{thm}
In particular, we have
\begin{cor} Let $\gamma$ be small enough. Then, the piecewise constant delayed Lasota–Wazewska model:
\begin{equation}\label{EeQqmodelx}
y'(t)=-\delta(t)y(t)+p(t)e^{-\gamma y([t])}\, ,
\end{equation}
has a unique compact almost automorphic solution.
\end{cor}

\subsection{Continuous $\mathbb{Z}$-almost automorphic solution of $1D$ heat equation}

In this subsection, we present a regularity result for the $1D$ heat equation in the space $\Z AA(\R;\R)\cap C(\R;\R)$. Note that, $\Z AA(\R;\R)\cap C(\R;\R)$ is a Banach space under the supremum norm. Let us consider the $1D$ heat equation
\begin{eqnarray}\label{HEap}
u_t-\bigtriangleup u&=& 0 \,\, \,\, {\rm in} \,\,\,\, (0,+\infty)\times \mathbb{R} \\
u(x,0^+)&=&u_0(x)\,\, \,\, {\rm in} \,\,\,\, \mathbb{R}\, .\nonumber
\end{eqnarray}

\begin{thm}\label{lastthhh}
If $u_0 \in \Z AA(\R;\R)\cap C(\R;\R)$; then, the solution $u$ of the heat equation (\ref{HEap}) satisfies 
$u(t,\cdot) \in \Z AA(\R;\R)\cap C(\R;\R),\, \, \forall t >0$ .
\end{thm}
\begin{proof}
It is well known that, the solution $u$ of the heat equation (\ref{HEap}) has the following integral representation 
\begin{equation}\label{Heat01}
u(t,x)=\int_{\mathbb{R}}\Phi(t,x-y)u_0(y)dy\, ,
\end{equation}
where $\Phi(t,x)$ is the fundamental solution with
\begin{equation}\label{HE01}
\int_{\mathbb{R}}\Phi(t,\xi)d\xi=1,\, \, \forall t >0\, .
\end{equation}
Now, using Proposition \ref{THF1} we conclude that $u(t,\cdot) \in \Z AA(\R;\R)\cap C(\R;\R),\, \, \forall t >0$ .
\end{proof}
Finally, we would like to comment that, as in \cite{chavez2021almost,ding2008asymptotically}, it is possible to study existence and uniqueness of the mild solution in the space $\Z AAA(\R^+;\mathbb{X})\cap C(\R^+;\mathbb{X})$ of the following integro-differential equation with nonlocal initial condition:
\begin{eqnarray}\label{EQ3}
  u'(t)&=&Au(t)+\int_0^t B(t-s)u(s)ds+f(t,u(t)),\ t\in \, \,  \R^+,\\
 u(0)&=&u_0+h(u),\label{EQ33}
\end{eqnarray}
where $u_0 \in \mathbb{X}$, $A$ and $B(t)_{t\geq 0}$ are dense and closed linear operators in $\mathbb{X}$; $A, B(\cdot), f, h$ satisfy some concrete conditions outlined, for instance, in \cite{chavez2021almost,ding2008asymptotically}. Of course, applications to the heat conduction in materials with memory can be improved.

\subsection{Open problem}
The Banach space $\Z AA(\R;\R)\cap C(\R;\R)$ was introduced in the previous subsection. Obviously, the inclusion 
$$AA(\R;\R) \subseteq  \Z AA(\R;\R)\cap C(\R;\R)\, ,$$ holds. 

The current problem is to identify a continuous $\mathbb{Z}$-almost automorphic function that does not fall under the category of classical almost automorphic functions. Therefore, the question arises: is the following strict inclusion 
$$AA(\R;\R) \varsubsetneq \Z AA(\R;\R)\cap C(\R;\R)\, ,$$
valid?.


\section*{Declarations of interest: none}

\section*{Acknowledgments}

The authors would like to express their gratitude to the anonymous reviewers for their careful reading and motivating suggestions, which have greatly contributed to the improvement of this work.

\section*{Funding}
Alan Ch\'avez was funded by contract 038-2021 Fondecyt-Per\'u and by CONCYTEC through the PROCIENCIA program under the E041-2023-01 competition, according to contract PE501082885-2023. Lenin Qui\~nones was funded by contract 038-2021 Fondecyt-Per\'u.







\begin{thebibliography}{00}
\bibitem{aftabizadeh1987oscillatory}
{\sc Aftabizadeh, A., Wiener, J., and Xu, J.-M.}
\newblock Oscillatory and periodic solutions of delay differential equations
  with piecewise constant argument.
\newblock {\em Proceedings of the American Mathematical Society 99}, 4 (1987),
  673--679.

\bibitem{akhmet2011method}
{\sc Akhmet, M.~U., Aru{\u{g}}aslan, D., and Y{\i}lmaz, E.}
\newblock Method of lyapunov functions for differential equations with
  piecewise constant delay.
\newblock {\em Journal of Computational and Applied Mathematics 235}, 16
  (2011), 4554--4560.

\bibitem{alonso2000almost}
{\sc Alonso, A., Hong, J., and Obaya, R.}
\newblock Almost periodic type solutions of differential equations with
  piecewise constant argument via almost periodic type sequences.
\newblock {\em Applied Mathematics Letters 13}, 2 (2000), 131--137.

\bibitem{araya2009almost}
{\sc Araya, D., Castro, R., and Lizama, C.}
\newblock Almost automorphic solutions of difference equations.
\newblock {\em Advances in Difference Equations 2009\/} (2009), 1--15.

\bibitem{bender1966some}
{\sc Bender, P.~R.}
\newblock {\em Some Conditions for the Existence of Recurrent Solutions to
  Systems of Ordinary Differential Equations, PhD Thesis}.
\newblock Iowa State University, Iowa, 1966.

\bibitem{berger2004almost}
{\sc Berger, A., Siegmund, S., and Yi, Y.}
\newblock On almost automorphic dynamics in symbolic lattices.
\newblock {\em Ergodic Theory and Dynamical Systems 24}, 3 (2004), 677--696.

\bibitem{03}
{\sc Bochner:, S.}
\newblock Curvature and {B}etti numbers in real and complex vector bundles.
\newblock {\em Univ. e Politec. Torino Rend. Sem. Mat. 15\/} (1955/56),
  225--253.

\bibitem{04}
{\sc Bochner, S.}
\newblock A new approach to almost periodicity.
\newblock {\em Proc. Nat. Acad. Sci. U.S.A. 48\/} (1962), 2039--2043.

\bibitem{05}
{\sc Bochner, S.}
\newblock Continuous mappings of almost automorphic and almost periodic
  functions.
\newblock {\em Proc. Nat. Acad. Sci. U.S.A. 52\/} (1964), 907--910.

\bibitem{busenberg1982models}
{\sc Busenberg, S., and Cooke, K.~L.}
\newblock Models of vertically transmitted diseases with sequential-continuous
  dynamics.
\newblock In {\em Nonlinear Phenomena in Mathematical Sciences}. Elsevier,
  1982, pp.~179--187.

\bibitem{campos2020barycentric}
{\sc Campos, J., and Tarallo, M.}
\newblock Barycentric solutions of linear almost periodic equations: Baire class and almost automorphy.
\newblock {\em Journal of Dynamics and Differential Equations 32}, 3 (2020),
  1475--1509.

\bibitem{chavez2014discontinuous}
{\sc Ch{\'a}vez, A., Castillo, S., and Pinto, M.}
\newblock Discontinuous almost automorphic functions and almost automorphic
  solutions of differential equations with piecewise constant argument.
\newblock {\em Electronic Journal of Differential Equations 2014}, 56 (2014),
  1--13.

\bibitem{chavez2014discon}
{\sc Ch{\'a}vez, A., Castillo, S., and Pinto, M.}
\newblock Discontinuous almost periodic type functions, almost automorphy of
  solutions of differential equations with discontinuous delay and
  applications.
\newblock {\em Electronic Journal of Qualitative Theory of Differential
  Equations 2014}, 75 (2014), 1--17.
  
  \bibitem{CHAKPPINTO2023BBMSS}
{\sc Ch\'avez, A., Khalil, K., Kosti\'c, M., and Pinto, M.}
\newblock Multi-dimensional almost automorphic type functions and applications.
\newblock {\em Bulletin of the Brazilian Mathematical Society, New Series,\/} 53, 3, (2021), 801--851.


\bibitem{CHAKPPINTO2023}
{\sc Ch\'avez, A., Khalil, K., Pereyra, A., and Pinto, M.}
\newblock Compact almost automorphic solutions to {P}oisson's and heat equations.
\newblock {\em Bulletin of the Malaysian Mathematical Sciences Society,\/} 47, 43 (2024). https://doi.org/10.1007/s40840-023-01637-5


\bibitem{chavez2021almost}
{\sc Ch{\'a}vez, A., Pinto, M., and Zavaleta, U.}
\newblock On almost automorphic type solutions of abstract integral equations,
  a Bohr-Neugebauer type property and some applications.
\newblock {\em Journal of Mathematical Analysis and Applications 494}, 1
  (2021), 124395.

\bibitem{cheban2020periodic}
{\sc Cheban, D., and Liu, Z.}
\newblock Periodic, quasi-periodic, almost periodic, almost automorphic, Birkhoff recurrent and Poisson stable solutions for stochastic differential
  equations.
\newblock {\em Journal of Differential Equations 269}, 4 (2020), 3652--3685.

\bibitem{chiu2019oscillatory}
{\sc Chiu, K.-S., and Li, T.}
\newblock Oscillatory and periodic solutions of differential equations with
  piecewise constant generalized mixed arguments.
\newblock {\em Mathematische Nachrichten 292}, 10 (2019), 2153--2164.

\bibitem{choquehuanca2023almost}
{\sc Choquehuanca, M., Mesquita, J., and Pereira, A.}
\newblock Almost automorphic solutions of second-order equations involving time
  scales with boundary conditions.
\newblock {\em Proceedings of the American Mathematical Society 151}, 03
  (2023), 1055--1070.

\bibitem{cooke1984retarded}
{\sc Cooke, K.~L., and Wiener, J.}
\newblock Retarded differential equations with piecewise constant delays.
\newblock {\em Journal of Mathematical Analysis and Applications 99}, 1 (1984),
  265--297.

\bibitem{cooke2006survey}
{\sc Cooke, K.~L., and Wiener, J.}
\newblock A survey of differential equations with piecewise continuous
  arguments.
\newblock In {\em Delay Differential Equations and Dynamical Systems:
  Proceedings of a Conference in honor of Kenneth Cooke held in Claremont,
  California, Jan. 13--16, 1990\/} (2006), Springer, pp.~1--15.

\bibitem{cuevas2012existence}
{\sc Cuevas, C., Henriquez, H.~R., and Lizama, C.}
\newblock On the existence of almost automorphic solutions of Volterra difference equations.
\newblock {\em Journal of Difference Equations and Applications 18}, 11 (2012),
  1931--1946.

\bibitem{dads2023massera}
{\sc Dads, E.~A., Es-Sebbar, B., and Lhachimi, L.}
\newblock On Massera and Bohr-Neugebauer type theorems for some almost automorphic differential equations.
\newblock {\em Journal of Mathematical Analysis and Applications 518}, 2
  (2023), 126761.

\bibitem{dads2010pseudo}
{\sc Dads, E.~A., and Lhachimi, L.}
\newblock Pseudo almost periodic solutions for equation with piecewise constant
  argument.
\newblock {\em Journal of Mathematical Analysis and Applications 371}, 2
  (2010), 842--854.

\bibitem{dai2020almost}
{\sc Dai, X.}
\newblock Almost automorphy and Veech's relations of dynamics on compact $T_2$-spaces.
\newblock {\em Journal of Differential Equations 269}, 3 (2020), 2580--2626.

\bibitem{de2012asymptotic}
{\sc de~Andrade, B., Cuevas, C., and Henr{\'\i}quez, E.}
\newblock Asymptotic periodicity and almost automorphy for a class of Volterra
  integro-differential equations.
\newblock {\em Mathematical Methods in the Applied Sciences 35}, 7 (2012),
  795--811.

\bibitem{diaganaBook2013almost}
{\sc Diagana, T.}
\newblock {\em Almost {A}utomorphic {T}ype and {A}lmost {P}eriodic {T}ype
  {F}unctions in {A}bstract {S}paces}, vol.~8.
\newblock Springer, 2013.

\bibitem{dimbour2011almost}
{\sc Dimbour, W.}
\newblock Almost automorphic solutions for differential equations with
  piecewise constant argument in a Banach space.
\newblock {\em Nonlinear Analysis: Theory, Methods \& Applications 74}, 6
  (2011), 2351--2357.

\bibitem{ding2017asymptotically}
{\sc Ding, H.-S., and Wan, S.-M.}
\newblock Asymptotically almost automorphic solutions of differential equations
  with piecewise constant argument.
\newblock {\em Open Mathematics 15}, 1 (2017), 595--610.

\bibitem{ding2008asymptotically}
{\sc Ding, H.-S., Xiao, T.-J., and Liang, J.}
\newblock Asymptotically almost automorphic solutions for some
  integrodifferential equations with nonlocal initial conditions.
\newblock {\em Journal of Mathematical Analysis and Applications 338}, 1
  (2008), 141--151.

\bibitem{es2016almost}
{\sc Es-Sebbar, B.}
\newblock Almost automorphic evolution equations with compact almost
  automorphic solutions.
\newblock {\em Comptes Rendus Mathematique 354}, 11 (2016), 1071--1077.

\bibitem{es2017almost}
{\sc Es-sebbar, B., and Ezzinbi, K.}
\newblock Almost periodicity and almost automorphy for some evolution equations
  using {F}avard’s theory in uniformly convex Banach spaces.
\newblock {\em Semigroup Forum\/} (2017), vol.~94, Springer, pp.~229--259.

\bibitem{es2020compact}
{\sc Es-sebbar, B., Ezzinbi, K., Fatajou, S., and Ziat, M.}
\newblock Compact almost automorphic weak solutions for some monotone
  differential inclusions: Applications to parabolic and hyperbolic equations.
\newblock {\em Journal of Mathematical Analysis and Applications 486}, 1
  (2020), 123805.

\bibitem{feng2020asymptotically}
{\sc Feng, Z., Wang, Y., and Ma, X.}
\newblock Asymptotically almost periodic solutions for certain differential
  equations with piecewise constant arguments.
\newblock {\em Advances in Difference Equations 2020}, 1 (2020), 1--22.

\bibitem{fink1968almost}
{\sc Fink, A.}
\newblock Almost automorphic and almost periodic solutions which minimize
  functionals.
\newblock {\em Tohoku Mathematical Journal, Second Series 20}, 3 (1968),
  323--332.

\bibitem{fink1969extensions}
{\sc Fink, A.}
\newblock Extensions of almost automorphic sequences.
\newblock {\em Journal of Mathematical Analysis and Applications 27}, 3 (1969),
  519--523.

\bibitem{hric2014construction}
{\sc Hric, R., and J{\"a}ger, T.}
\newblock A construction of almost automorphic minimal sets.
\newblock {\em Israel Journal of Mathematics 204\/} (2014), 373--395.

\bibitem{kashkynbayev2020global}
{\sc Kashkynbayev, A., and Koptleuova, D.}
\newblock Global dynamics of tick-borne diseases.
\newblock {\em Mathematical Biosciences and Engineering: MBE 17}, 4 (2020),
  4064--4079.

\bibitem{kupper2002quasi}
{\sc K{\"u}pper, T., and Yuan, R.}
\newblock On quasi-periodic solutions of differential equations with piecewise
  constant argument.
\newblock {\em Journal of Mathematical Analysis and Applications 267}, 1
  (2002), 173--193.

\bibitem{milnes1977almost}
{\sc Milnes, P.}
\newblock Almost automorphic functions and totally bounded groups.
\newblock {\em The Rocky Mountain Journal of Mathematics 7}, 2 (1977),
  231--250.

\bibitem{16}
{\sc N'Gu\'er\'ekata, G.}
\newblock {\em Almost Periodic and Almost Automorphic Functions in Abstract
  Spaces}.
\newblock Springer Cham, 2021.

\bibitem{qi2022piecewise}
{\sc Qi, L., and Yuan, R.}
\newblock Piecewise continuous almost automorphic functions and {F}avard’s
  theorems for impulsive differential equations in honor of Russell Johnson.
\newblock {\em Journal of Dynamics and Differential Equations\/} (2022), 1--43.

\bibitem{rong1997existence}
{\sc Rong, Y., and Jialin, H.}
\newblock The existence of almost periodic solutions for a class of
  differential equations with piecewise constant argument.
\newblock {\em Nonlinear Analysis: Theory, Methods \& Applications 28}, 8
  (1997), 1439--1450.

\bibitem{shen2000oscillatory}
{\sc Shen, J., and Stavroulakis, I.}
\newblock Oscillatory and nonoscillatory delay equations with piecewise
  constant argument.
\newblock {\em Journal of Mathematical Analysis and Applications 248}, 2
  (2000), 385--401.

\bibitem{terras1970almost}
{\sc Terras, R.}
\newblock {\em Almost automorphic functions on topological groups}.
\newblock University of Illinois at Urbana-Champaign, 1970.

\bibitem{van2007almost}
{\sc Van~Minh, N., and Dat, T.~T.}
\newblock On the almost automorphy of bounded solutions of differential
  equations with piecewise constant argument.
\newblock {\em Journal of Mathematical Analysis and Applications 326}, 1
  (2007), 165--178.

\bibitem{18}
{\sc Veech, W.~A.}
\newblock Almost automorphic functions.
\newblock {\em Proceedings of the National Academy of Sciences 49}, 4 (1963),
  462--464.

\bibitem{veech1965almost}
{\sc Veech, W.~A.}
\newblock Almost automorphic functions on groups.
\newblock {\em American Journal of Mathematics 87}, 3 (1965), 719--751.

\bibitem{xia2007existence}
{\sc Xia, Y., Huang, Z., and Han, M.}
\newblock Existence of almost periodic solutions for forced perturbed systems
  with piecewise constant argument.
\newblock {\em Journal of Mathematical Analysis and Applications 333}, 2
  (2007), 798--816.

\bibitem{yuan2010favard}
{\sc Yuan, R.}
\newblock On {F}avard's theorems.
\newblock {\em Journal of Differential Equations 249}, 8 (2010), 1884--1916.

\bibitem{zaki1974almost}
{\sc Zaki, M.}
\newblock Almost automorphic solutions of certain abstract differential
  equations.
\newblock {\em Annali di Matematica Pura ed Applicata 101}, 1 (1974), 91--114.



\end{thebibliography}






\end{document}